\documentclass[reqno]{amsart} 
\usepackage{amsfonts, amsmath, amsthm, amssymb, latexsym, graphicx}

\theoremstyle{plain}
\newtheorem{theorem}{Theorem}[section]

\theoremstyle{definition}

\numberwithin{equation}{section}
\newcommand{\cond}[1]{\, #1 \vert \,}
\newcommand{\diff}{\,\mathrm{d}}
\newcommand{\mean}{\, \mathrm{E} \,}

\title{Martingale estimation functions for Bessel processes}

\author{Nicole Hufnagel, Jeannette H.C. Woerner} 
\address{Fakult\"at Mathematik, Technische Universit\"at Dortmund,
          Vogelpothsweg 87,
          D-44221 Dortmund, Germany}
\email{nicole.hufnagel@math.tu-dortmund.de, jeannette.woerner@math.tu-dortmund.de}


\begin{document}
\date{\today}

\begin{abstract}
In this paper we derive martingale estimating functions for the dimensionality parameter of a Bessel process based on the eigenfunctions of the diffusion operator. Since a Bessel process is non-ergodic and the theory of martingale estimating functions is developed for ergodic diffusions, we use the space-time transformation of the Bessel process and formulate our results for a modified Bessel process. We deduce consistency, asymptotic normality and discuss optimality. It turns out that the martingale estimating function based of the first eigenfunction of the modified Bessel process coincides with the linear martingale estimating function for the Cox Ingersoll Ross process. Furthermore, our results may also be applied to estimating the multiplicity parameter of a one-dimensional Dunkl process. 
\end{abstract}

\maketitle

\section{Introduction} 
Martingale estimating functions introduced in \cite{bibbysorensen1995} provide a well-established method for inference in discretely observed diffusion processes, when the likelihood function is unknown or too complicated. The idea behind martingale estimating functions is to provide a simple approximation of the true likelihood, which forms a martingale and hence leads under suitable regularity assumptions to consistent and asymptotically normal estimators. One way of approximating the likelihood function is by Taylor expansion leading to linear and quadratic martingale estimating functions, cf. \cite{bibbysorensen1995}. Another possibility is to use the eigenfunctions of the associated diffusion operator, cf. \cite{kessler1999}. In this context a suitable optimality concept was introduced by \cite{godambeheyde1987} and \cite{heyde1988}. For a general theory of asymptotic statistics for diffusion processes we refer e.g. to \cite{Hoepfner2014}.

Our aim in this paper is to estimate the dimensionality or index parameter $\vartheta\in\Theta \subset (-\frac{1}{2},\infty)$ of a classical one-dimensional Bessel process given by the stochastic differential equation
\begin{align}
	\label{eq:sdgl_BP}
	\left\{ 
		\begin{array}{ll}
			\diff Y_t &=\diff B_t+\left(\vartheta+\frac{1}{2}\right)\frac{1}{Y_t} \diff t,\\
			Y_0&=y_0>0,
		\end{array}
	\right.
\end{align} 
where $B$ denotes a standard Brownian motion. Since a Bessel process is non-ergodic, we transform it into a stationary and ergodic process by adding a mean reverting term with speed of mean reversion $\alpha>0$ in the drift, which we call modified Bessel process in the following. The two process are then related by the well-known space-time transformation of a Bessel process. Since the eigenfunctions of the associated diffusion operator of the modified Bessel process are known, we base our martingale estimation function on these eigenfunctions and follow the lines of \cite{kessler1999}. 

For the estimating function based on the first eigenfunction we obtain an explicit formula for the estimator, which only depends quadratically on the observations. We see that the estimator coincides with the one of a linear martingale estimation function for the Cox Ingersoll Ross process, which is the square of the modified Bessel process. We discuss optimality in the sense of Godambe and Heyde. Note that in \cite{overbeck} also local asymptotic normality of the Cox Ingersoll Ross process for $\vartheta>0$ was established.

Furthermore, we consider martingale estimating functions based on the first two eigenfunctions and discuss the improvement of the asymptotic variance. In this case we do not get an explicit estimator anymore.

Note that our results for the Bessel process may also be used to estimate the multiplicity parameter $k$ of a one-dimensional Dunkl process, a special jump diffusion given by the generator 
$$ L_k u(x)=u''(x)+k\left(\frac{2}{x}u'(x)+\frac{u(-x)-u(x)}{x^2}\right),\;\; k\geq 0.$$
By the last term in the generator we see that the associated process possesses jumps due to a reflection, which lead to a sign change.
Hence, the modulus of this Dunkl process is a Bessel process with dimensionality parameter $k-1/2$, cf. \cite{CGY}. For the Dunkl process the multiplicity parameter is of special interest, since it determines the jump activity, namely for $k\geq 0$ a Dunkl process has a finite jump activity, whereas for $k<1/2$ we have infinite jump activity.

The paper is organised as follows: in Section 2 we collect the basic facts on the processes, Section 3 is devoted to martingale estimation functions based on the first eigenfunction and Section 4 to estimators based on two eigenfunctions.

\section{Basic results on Bessel processes and a stationary modification}
In this section we introduce the basic results on the underlying diffusions, which we will need in the following for the theory of martingale estimation functions. Our aim is to estimate the parameter $\vartheta\in\Theta \subset (-\frac{1}{2},\infty)$ of a classical one-dimensional Bessel process. Since a Bessel process is non-ergodic and most results on parameter estimation for diffusions are developed for ergodic diffusions, we start by introducing a modification of a Bessel process which is ergodic.

We consider the stochastic differential equation
\begin{align}
	\label{eq:sdgl_korrigierter_BP}
	\left\{ 
		\begin{array}{ll}
			\diff X_t &=\diff B_t+\left[\left(\vartheta+\frac{1}{2}\right)\frac{1}{X_t}-\alpha X_t\right] \diff t,\\
			X_0&=x_0>0
		\end{array}
	\right.
\end{align} 
for a Brownian motion $B$, some fixed $\alpha >0$ and the parameter of interest $\vartheta\in\Theta \subset (-\frac{1}{2},\infty)$. The equation \eqref{eq:sdgl_korrigierter_BP} is similar to the equation defining a Bessel process except for the drift term $-\alpha X_t \diff  t$, which we add to ensure ergodicity and stationarity.

In order to determine the density of $(X_t)_{t\geq 0}$, we consider the space time transformation
\begin{align}
	\label{eq:transformation_explizit}
	X_t:=\exp(-\alpha t)Y_{\frac{\exp(2\alpha t)-1}{2\alpha }}
\end{align}
for a Bessel process $(Y_t)_{t\geq 0}$ with index $\vartheta$, which immediately follows by It\^{o}'s formula. 
Therefore, we derive the distribution of $(X_t)_{t\geq 0}$ by using the well-known distribution of the Bessel process $(Y_t)_{t\geq 0}$, namely 
\begin{align*}
	P(Y_t\leq z|Y_0=x)=\frac{2}{(2t)^\vartheta \Gamma(\vartheta+1)}\int_0^z j_\vartheta\left(\frac{ixy}{t}\right) e^{-\frac{x^2+y^2}{2t}}y^{2\vartheta+1}\diff y; \quad x, z >0,
\end{align*}
where 
\begin{align*}
&j_\vartheta(z):=\frac{\Gamma(\vartheta+1)}{\Gamma(\vartheta+\frac{1}{2})\Gamma(\frac{1}{2})}\int_{-1}^{1}e^{isz}(1-s^2)^{\vartheta-\frac{1}{2}}\diff s
\end{align*}
is the Bessel function with index $\vartheta$ (see for instance \cite{ito1974diffusion}). Hence, we obtain
\begin{align*}
&P(X_t\leq z|X_0=x)\stackrel{\eqref{eq:transformation_explizit}}{=}P(Y_{\frac{\exp(2\alpha t)-1}{2\alpha }}\leq \exp(\alpha t)z|Y_0=x)\\
&=  C_{\vartheta,\alpha,t}\int_0^{z} j_\vartheta\left(ixy\frac{2\alpha  \exp(\alpha t)}{\exp(2\alpha t)-1}\right)
\exp\left(-\alpha\frac{x^2+y^2\exp(2\alpha t)}{\exp(2\alpha t)-1}\right)y^{2\vartheta+1} \diff y
\end{align*}
with
\begin{align*}
	C_{\vartheta,\alpha,t}:=\frac{2\alpha^\vartheta (\exp(2\alpha t))^{\vartheta+1}}{\Gamma(\vartheta+1)(\exp(2\alpha t)-1)^\vartheta }.
\end{align*}
We denote the density of $X_{\Delta}$ with starting point $x$ by $p_\vartheta(x,\cdot ,\Delta)$ and the distribution of $X_\Delta$ by $P_{\vartheta}$. In the following, we check that $(X_t)_{t\geq 0}$ is indeed stationary and ergodic and determine the invariant measure.
The density of the scale measure for a fixed $\xi\in(0,\infty)$ is defined as
	\begin{align*}
	s(x)&:=\exp\left( -2\int_{\xi}^x\left(\vartheta+\frac{1}{2}\right)\frac{1}{y}-\alpha 	 		      
	y\diff y \right)\\
	&=\left(\frac{x}{\xi}\right)^{-(2\vartheta +1)} e^{\alpha (x^2-\xi^2)}.
	\end{align*}
 Note that, due to the singularity in the drift, we initially have to consider some positive interior point $\xi$.

From this we may deduce that $(X_t)_{t\geq 0}$ is ergodic as we see that the conditions
\begin{align*}
	&\int_0^\xi s(x) \diff x=\infty,\quad \int_\xi^\infty s(x)\diff x=\infty \quad \textrm{and} \quad \int_0^\infty \frac{1}{s(x)}\diff x<\infty
	\end{align*}
	are satisfied.
	
As the invariant measure is defined via the scale measure $m(\diff x):=\frac{1}{s(x)}\diff x$, we obtain by a straight forward calculation that
the density of the invariant probability measure is given by
	\begin{align*}
	\mu_\vartheta(x)=\frac{2\alpha^{\vartheta +1}}{\Gamma(\vartheta +1)}x^{2\vartheta+1}e^{-\alpha x^2}
	\end{align*}
	on $(0,\infty)$ with respect to the Lebesgue measure.
	
For the calculation of the asymptotic variance we will need	
the symmetric distribution $Q_\Delta^\vartheta$ of two consecutive observations $X_{(i-1)\Delta}$ and $X_{i\Delta}$ on $(0,\infty)^2$. It is given by
\begin{align*}
&Q_\Delta^\vartheta(\diff x,\diff y)=\mu_\vartheta (x) p(x,y,\Delta)\diff x\diff y\\
&=C_\vartheta j_\vartheta\left(ixy\frac{2\alpha  \exp(\alpha \Delta)}{\exp(2\alpha \Delta)-1}\right)
\exp\left(-\frac{\alpha\exp(2\alpha\Delta)}{\exp(2\alpha\Delta)-1}(x^2+y^2)\right)(xy)^{2\vartheta+1} \diff y \diff x
\end{align*}
with 
\begin{align*}
	C_\vartheta:=\frac{4\alpha^{2\vartheta} (\exp(2\alpha \Delta))^{\vartheta+1}}{\Gamma(\vartheta+1)^2(\exp(2\alpha \Delta)-1)^\vartheta }.
\end{align*}

\section{Martingale estimating functions based on eigenfunctions}
In this section we proceed similarly to \cite{bibbysorensen1995} and \cite{kessler1999}  to construct martingale estimation functions for our parameter of interest $\vartheta$. The concepts in these papers are based on ergodic diffusions. As Bessel processes are non-ergodic we constructed the ergodic and stationary version in \eqref{eq:sdgl_korrigierter_BP}. Let $X_{\Delta},\dots ,X_{n\Delta}$ be discrete observations of the process.  
	We consider the eigenfunctions of the generator 
	$$L_\vartheta f(x) =\left[\left(\vartheta +\frac{1}{2}\right)\frac{1}{x}-\alpha x\right]f'(x)+\frac{1}{2}f''(x),$$
	which are the solutions of $L_\vartheta \phi_\eta=-\lambda_\eta \phi_\eta$ given by 
	\begin{align*}
		\lambda_\eta=2\alpha \eta, \quad \phi_\eta(x,\vartheta)= \sum_{k=0}^{\eta} \frac{(-\eta)_k}{(\vartheta + 1 )_k k!} (\alpha x^2)^k,\quad \eta \in \mathbb{N}
	\end{align*}
	with the Pochhammer symbols $(x)_k:=\frac{\Gamma(x+n)}{\Gamma(x)}=x (x+1) \dots (x+k-1)$.
	According to \cite{kessler1999}, the property 
	\begin{align*}
		\int_0^\infty (\phi_\eta'(x,\vartheta))^2\mu_\vartheta(\diff x)= \frac{2\alpha^{\vartheta +1}}{\Gamma(\vartheta +1)}\int_0^\infty (\phi_\eta'(x,\vartheta))^2x^{2\vartheta+1}e^{-\alpha x^2} \diff x<\infty
	\end{align*}
	for the polynomials $\phi_\eta$ is sufficient to deduce 
	$$\mathrm{E}_\vartheta(\phi_\eta(X_{i\Delta},\vartheta)|X_{(i-1)\Delta})=e^{-\lambda_\eta \Delta}\phi_\eta(X_{(i-1)\Delta},\vartheta)$$
	by It\^{o}'s formula. Consequently, we may use the general theory on estimators based on eigenfunctions given in \cite{kessler1999}. However, in our case we may calculate the involved quantities and obtain explicit results. For the first eigenfunction $\phi_1(x,\vartheta)=1-\frac{\alpha x^2}{\vartheta+1}$ we consider the estimator based on the martingale estimating function
	\begin{align*}
		G_n(\vartheta)&=\sum_{i=1}^{n}(\phi_1(X_{i\Delta},\vartheta)-e^{-\lambda_1\Delta}\phi_1(X_{(i-1)\Delta},\vartheta) )\\
		&= n(1-e^{-2\alpha \Delta})+\sum_{i=1}^n\left(e^{-2\alpha \Delta}\frac{\alpha X_{(i-1)\Delta}^2}{\vartheta+1}-\frac{\alpha X_{i\Delta}^2}{\vartheta+1}\right).
	\end{align*}
	The unique solution of $G_n(\widehat{\vartheta}_n)=0$ is
	\begin{align*}
		\widehat{\vartheta}_n=\frac{\alpha \sum_{i=1}^{n} (X_{i\Delta }^{2}-X_{(i-1)\Delta}^{2} e^{-2\alpha \Delta})}{n(1-e^{-2\alpha \Delta})}-1.
	\end{align*}
	Now, we may deduce consistency and asymptotic normality along the same lines as for general martingale estimating functions.
	\begin{theorem}
		\label{theo:asymptotic_estimator_mean_value}
		For every true value $\vartheta_0 \in \Theta \subset (-\frac{1}{2},\infty)$ we have 
		\begin{itemize}
			\item[(i)] $\widehat{\vartheta}_n\to \vartheta_0$ in probability and
			\item[(ii)] $\sqrt{n}(\widehat{\vartheta}_n-\vartheta_0)\to N(0,\sigma^2(\vartheta_0))$ in distribution
		\end{itemize}
		under $P_{\vartheta_0}$ with $\sigma^2(\vartheta_0):=(\vartheta_0+1)\frac{1+e^{-2\alpha\Delta}}{1-e^{-2\alpha \Delta}}.$
	\end{theorem}
\begin{proof}[Proof:]
	We define 
	\begin{align*}
		g(x,y,\vartheta):= 1-\frac{\alpha y^2}{\vartheta+1}-e^{-2\alpha\Delta}\left(1-\frac{\alpha x^2}{\vartheta+1}\right)
	\end{align*}
	a continuously differentiable function with respect to $\vartheta$. The absolute value of the derivative
	\begin{align*}
		\frac{\partial}{\partial\vartheta} g(x,y,\vartheta)= \frac{\alpha}{(\vartheta+1)^2}(y^2-e^{-2\alpha \Delta}x^2)
	\end{align*}
	is dominated by $4\alpha(y^2+e^{-2\alpha \Delta}x^2)$, which is independent of $\vartheta$ and square integrable with respect to $Q_\Delta^{\vartheta_0}$. Moreover, the symmetry in $x$ and $y$ of the density of $Q_\Delta^{\vartheta_0}$ implies
	\begin{align*}
		&\int_0^\infty \int_0^\infty \frac{\partial}{\partial\vartheta} g(x,y,\vartheta_0) Q_\Delta^{\vartheta_0}(\diff x,\diff y)\\
		&=\underbrace{\frac{\alpha}{(\vartheta_0+1)^2}(1-e^{-2\alpha \Delta})}_{>0} \underbrace{\int_0^\infty \int_0^\infty x^2 Q_\Delta^{\vartheta_0}(\diff x,\diff y)}_{>0}\not =0,
	\end{align*}
	which completes the proof of (i) and (ii) according to \cite[Theorem 4.3]{kessler1999}.

	Due to \cite{kessler1999}, the asymptotic variance is given by $\sigma^2(\vartheta_0)=\frac{v(\vartheta_0)}{f^2(\vartheta_0)}$ with the functions
	\begin{align*}
		f(\vartheta_0)&:=\int_0^\infty \int_0^\infty \frac{\partial}{\partial\vartheta} g(x,y,\vartheta_0) Q_\Delta^{\vartheta_0}(\diff x,\diff y)\\
		&=\frac{\alpha}{(\vartheta_0+1)^2}(1-e^{-2\alpha \Delta}) \int_0^\infty \int_0^\infty x^2 Q_\Delta^{\vartheta_0}(\diff x,\diff y),\\
		v(\vartheta_0)&:=\int_0^\infty \int_0^\infty  g^2(x,y,\vartheta_0) Q_\Delta^{\vartheta_0}(\diff x,\diff y)\stackrel{!}{=}\frac{1-e^{-4\alpha \Delta}}{\vartheta_0+1}.
	\end{align*}
	Because of the symmetry of $Q_\Delta^{\vartheta_0}$ and
	\begin{align*}
		g^2(x,y,\vartheta) = (1-e^{-2\alpha \Delta})^2+\frac{\alpha^2}{(\vartheta+1)^2}y^4+\frac{\alpha^2 e^{-4\alpha \Delta}}{(\vartheta+1)^2}x^4-(1-e^{-2\alpha \Delta})\frac{2\alpha}{\vartheta+1}y^2\\
		+(1-e^{-2\alpha \Delta}) \frac{2\alpha e^{-2\alpha \Delta }}{\vartheta+1}x^2-\frac{2\alpha^2 e^{-2\alpha \Delta}}{(\vartheta+1)^2}x^2y^2,
	\end{align*}
	we get
	\begin{multline*}
		v(\vartheta_0)=(1-e^{-2\alpha \Delta})^2\left(1-\frac{2\alpha}{\vartheta_0+1}\int_0^\infty \int_0^\infty x^2 Q_\Delta^{\vartheta_0}(\diff x,\diff y)\right)\\
		\shoveright{+\frac{\alpha^2(1+e^{-4\alpha \Delta})}{(\vartheta_0+1)^2}\int_0^\infty \int_0^\infty x^4 Q_\Delta^{\vartheta_0}(\diff x,\diff y)}\\
		-\frac{2\alpha^2 e^{-2\alpha \Delta}}{(\vartheta_0+1)^2}\int_0^\infty \int_0^\infty x^2y^2 Q_\Delta^{\vartheta_0}(\diff x,\diff y).
	\end{multline*}
	Furthermore, we can calculate
	\begin{align*}
		\int_0^\infty \int_0^\infty x^{2n}Q_\Delta^{\vartheta_0}(\diff x,\diff y)&=	\int_0^\infty \int_0^\infty x^{2n}\mu_{\vartheta_0} (x) p(x,y,\Delta)\diff x\diff y \\
		&=\int_0^\infty x^{2n}\mu_{\vartheta_0} (x)\diff x=\frac{\Gamma(n+{\vartheta_0}+1)}{\alpha^n\Gamma({\vartheta_0}+1)}.
	\end{align*}
	By using an explicit formula of the conditional mean, we conclude
	\begin{align*}
		\int_{0}^\infty	\int_{0}^\infty x^2y^2 Q_\Delta^{\vartheta_0}(\diff x,\diff y)&=	\int_{0}^\infty \int_{0}^\infty x^2y^2 \mu_{\vartheta_0}(x) p(x,y,\Delta)\diff y \diff x \\
		&=	\int_{0}^\infty x^2 \mean(X^2_{i\Delta} \cond{} X_{(i-1)\Delta}=x) \mu_{\vartheta_0}(x) \diff x\\
		&= \int_{0}^{\infty } x^2\left(x^2e^{-2\alpha \Delta}-\frac{\vartheta_0+1}{\alpha}(e^{-2\alpha\Delta}-1)\right)\mu_{\vartheta_0}(x) \diff x\\
		&= \frac{(\vartheta_0+1)^2}{\alpha^2}+e^{-2\alpha\Delta}\frac{\vartheta_0+1}{\alpha^2}.
	\end{align*}
	Applying these formulas we establish
	\begin{align*}
		\sigma^2(\vartheta_0)&=\frac{v(\vartheta_0)}{f^2(\vartheta_0)}=(\vartheta_0+1)\frac{1+e^{-2\alpha\Delta}}{1-e^{-2\alpha \Delta}}.
	\end{align*}
\end{proof}
Let us discuss the results. Looking at the asymptotic variance we see that it decreases when $\alpha \Delta $ is increasing. This seems surprisingly at the first glance, since it implies that the asymptotic variance decreases when the distance between observations increases, as we keep the mean reverting parameter $\alpha$ fixed. On the other hand, keeping in mind that equidistant observations for the stationary version of the Bessel process means the distance between two observations of the underlying Bessel process is exponentially growing. This leads to a fast growing observation interval, capturing the non-stationary behaviour of the original Bessel process. Furthermore, we see that the asymptotic variance tends to infinity as the mean-reverting parameter tends to zero.

Having a closer look at the estimator, we see that it only depends on the square of the observations, hence we could reformulate our problem and consider the squared process $Y_t:=X_t^2$. It\^{o}'s formula yields
	\begin{align*}
		\diff Y_t= 2\sqrt{Y_t}\diff B_t+(2\vartheta+2-2\alpha Y_t)\diff t,
	\end{align*}
an equation describing a Cox Ingersoll Ross process. We consider now the canonical linear martingale estimating function
	\begin{align*}
		\widetilde{G}_n(\vartheta)&:=\sum_{i=1}^n (Y_{i\Delta}- \mathrm{E}(Y_{i\Delta}|Y_{(i-1)\Delta}))\\
		&=\sum_{i=1}^n (Y_{i\Delta}- Y_{(i-1)\Delta}e^{-2\alpha \Delta}+\frac{\vartheta+1}{\alpha}(e^{-2\alpha\Delta}-1))\\
		&=-\frac{\vartheta+1}{\alpha}G_n(\vartheta).
	\end{align*}
For $\vartheta>-\frac{1}{2}$ the unique solution of $\widetilde{G}_n(\widehat{\vartheta}_n)=0$ is again
	\begin{align*}
		\widehat{\vartheta}_n=\frac{\alpha \sum_{i=1}^{n} (X_{i\Delta }^{2}-X_{(i-1)\Delta}^{2} e^{-2\alpha \Delta})}{n(1-e^{-2\alpha \Delta})}-1.
	\end{align*} 
Hence, we see that the two estimators coincide. In \ref{theo:asymptotic_estimator_mean_value} we have already established the consistence and asymptotic normality of $\widehat{\vartheta}_n$. 

The next step is to search for the optimal asymptotic variance by using estimators of the form 
	\begin{align*}
		\sum_{i=1}^{n} g_{i-1}(\vartheta) \left(X_{i\Delta}^2-X_{(i-1)\Delta}^2e^{-2\alpha\Delta}+\frac{\vartheta+1}{\alpha}(e^{-2\alpha\Delta}-1)\right),
	\end{align*}
where $g_{i-1}$ is $\sigma(X_\Delta,\dots, X_{n\Delta})$ measurable and continuously differentiable.
Considering this second approach via linear martingale estimating functions for the squared process, allows us easily to determine this optimal estimator, cf. \cite{heyde1988}, \cite{godambeheyde1987}.  By \cite[(2.10)]{bibbysorensen1995} the optimal estimator is given by
	\begin{align*}
		g_{i-1}(\vartheta):= \frac{\frac{\diff}{\diff \vartheta} \mean (X_{i\Delta}^2 \cond{} X_{(i-1)\Delta})}{\phi(X_{(i-1)\Delta}^2,\vartheta)}=\frac{1}{\frac{\vartheta+1}{\alpha}(1-e^{-2\alpha\Delta})+2X_{(i-1)\Delta}^2e^{-2\alpha\Delta}},
	\end{align*}
where $\phi$ is the conditional variance of $X_{i\Delta}$. Unfortunately, the equation
	\begin{multline*}
		\sum_{i=1}^{n} \frac{1}{\frac{\vartheta+1}{\alpha}(1-e^{-2\alpha\Delta})+2X_{(i-1)\Delta}^2e^{-2\alpha\Delta}}\times \\
		\times \left(X_{i\Delta}^2-X_{(i-1)\Delta}^2e^{-2\alpha\Delta}+\frac{\vartheta+1}{\alpha}(e^{-2\alpha\Delta}-1)\right)=0
	\end{multline*}
is not explicitly solvable with respect to $\vartheta$. However, we can nevertheless determine the improvement in the asymptotic variance. Following again the same lines as \cite[Theorem 3.2]{bibbysorensen1995}, we have to establish the finiteness of 
	\begin{align*}
		\mean_{\mu_{\vartheta_0}}\left( g_{i-1}(\vartheta_0)\frac{\diff}{\diff \vartheta} \mean (X_{i\Delta}^2 \cond{} X_{(i-1)\Delta})\right) &=\mean_{\mu_{\vartheta_0}} \left(\frac{1}{\vartheta_0+1+\frac{2\alpha e^{-2\alpha \Delta}}{1-e^{-2\alpha\Delta}} X_{(i-1)\Delta}^2}\right)\\
		&< \frac{1}{\vartheta_0+1},
	\end{align*}
the reciprocal of the asymptotic variance. Consequently, we can deduce that a lower bound of the optimal variance is given by $\vartheta_0 +1$.
\begin{figure}[h]
			\centering 
			\includegraphics[scale=0.35]{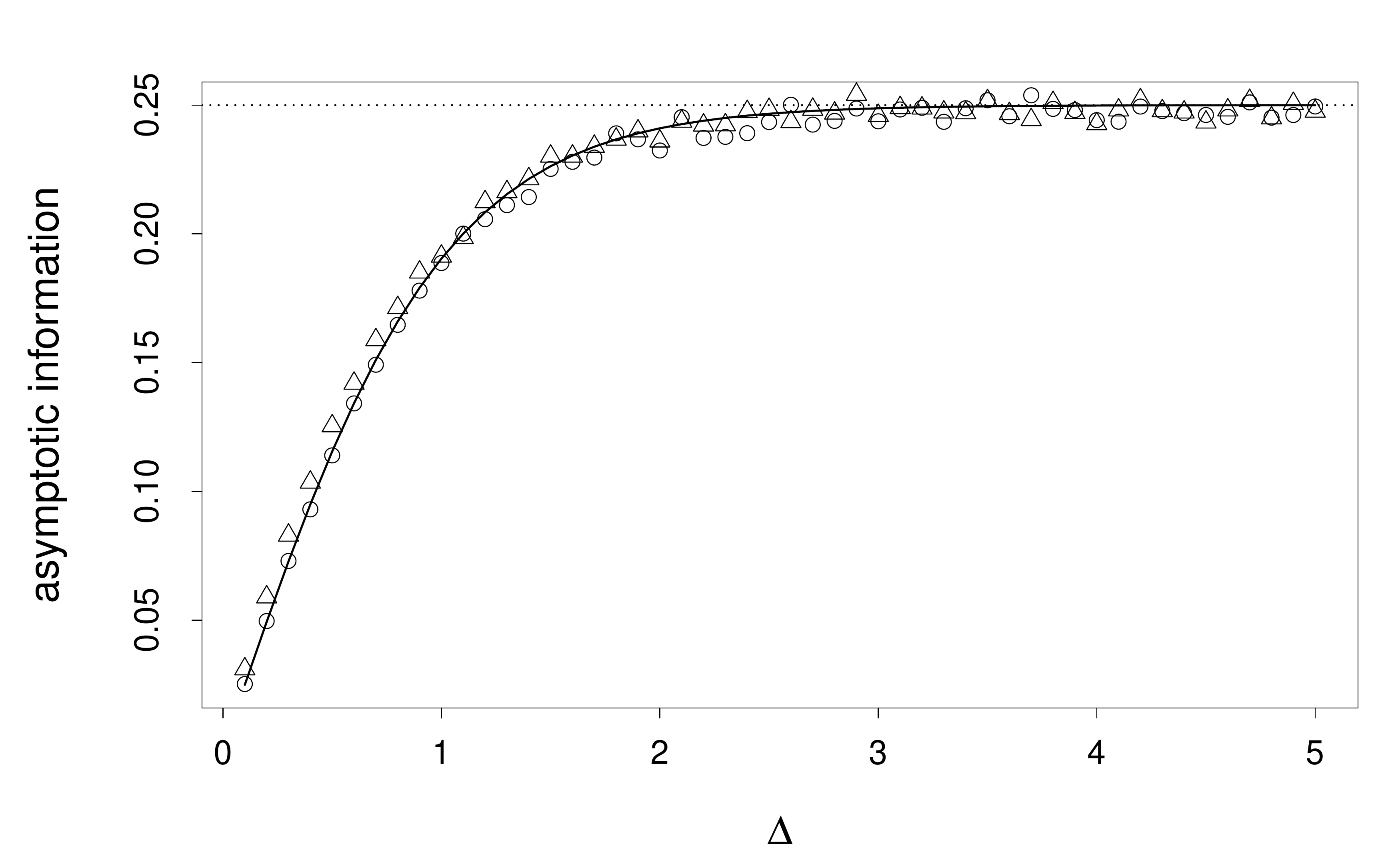}
			\caption{The asymptotic behavior for $\alpha=1, x_0=0.1, \vartheta=3$.}
			\label{fig:asymptotic_behavior}
	\end{figure}

Figure \ref{fig:asymptotic_behavior} shows the asymptotic behaviour of the 10.000 simulated optimal estimator (triangles) and $\widehat{\vartheta}_n$ (dots) for $n=1.000$. The solid line corresponds to the calculated asymptotic information of $\widehat{\vartheta}_n$ in \ref{theo:asymptotic_estimator_mean_value}. The dotted line represents our computed bound above.  As the lines nearly touch around $\Delta=3$, the improvement of the optimal estimator quickly tends to zero.  Starting from the value $\Delta=1$ the simulated asymptotic information is almost the same for both estimators. Beforehand, the improvement is clearly visible but we do not want to maintain such a high variance as we can choose the value of $\alpha\Delta$ such that the asymptotic variance is close to the lower bound. 

 We take a closer look at the asymptotic variance of $\widehat{\vartheta}_n$, which decreases monotonously in $\alpha \Delta$:
\begin{align*}
	\lim\limits_{\alpha \Delta\to \infty} (\vartheta_0+1)\frac{1+e^{-2\alpha\Delta}}{1-e^{-2\alpha\Delta}}=\vartheta_0+1.
\end{align*}
Due to the fast convergence to the lower bound $\vartheta_0+1$, we can for practical purposes restrict ourselves to the estimator $\widehat{\vartheta}_n$ and hence have an explicit estimator.

\section{Estimator based on two eigenfunctions}

Now, we try to improve the asymptotic variance further by considering martingale estimating functions based on two eigenfunctions. Yet, this approach suffers from the drawback that we do not get an explicit estimator anymore.

We consider
\begin{align*}
	H_n(\vartheta):=\sum_{i=1}^{n} \sum_{j=1}^{2} \beta_j(\vartheta)\left(\phi_j(X_{i\Delta,}\vartheta)-e^{-\lambda_j(\vartheta)\Delta} \phi_j(X_{(i-1)\Delta},\vartheta)\right),
\end{align*}
where $\beta_1$ and $\beta_2$ are continuously differentiable functions only depending on $\vartheta$. Under suitable conditions on the interplay between the weights $\beta_i$ and the eigenfunctions, we can easily achieve a consistent and asymptotic normal estimator. 
\begin{theorem}
	\label{theo:asymptotic_variance}
	If for every $\vartheta\in \Theta$ 
	$$f(\beta_1,\beta_2,\vartheta):=\beta_1(\vartheta)\frac{1-e^{-2\alpha\Delta}}{\vartheta+1}+\beta_2(\vartheta)\frac{1-e^{-4\alpha\Delta}}{(\vartheta+1)(\vartheta+2)}\not=0$$
	is satisfied, then there exists a solution of $H_n(\widehat{\vartheta}_{n,2})=0$ with a probability tending to one as $n\to \infty$ under $P_{\vartheta_0}$. Furthermore, for every true value $\vartheta_0 \in \Theta \subset (-\frac{1}{2},\infty)$ we have 
	\begin{itemize}
		\item[(i)] $\widehat{\vartheta}_{n,2}\to \vartheta_0$ in probability and
		\item[(ii)] $\sqrt{n}(\widehat{\vartheta}_{n,2}-\vartheta_0)\to N\left(0,\frac{v(\beta_1,\beta_2,\vartheta_0)}{f^2(\beta_1,\beta_2,\vartheta_0)}\right)$ in distribution
	\end{itemize}
	under $P_{\vartheta_0}$ with
	\begin{align*}
		v(\beta_1,\beta_2,\vartheta_0):=\beta_1^2(\vartheta_0)\frac{1-e^{-4\alpha\Delta}}{\vartheta_0+1}+\beta_2^2(\vartheta_0)\frac{2-2e^{-8\alpha\Delta}}{(\vartheta_0+1)(\vartheta_0+2)}.
	\end{align*}
\end{theorem}
\begin{proof}[Proof:]	
	As  by the assumption $f(\cdot, \cdot, \vartheta)\not = 0$ for every $\vartheta\in \Theta$, we conclude $\beta_1(\vartheta)\not =0$ or $\beta_2(\vartheta)\not = 0$ and consequently $v(\cdot, \cdot,\vartheta)\not = 0 $ for every $\vartheta\in\Theta$. Using again \cite{kessler1999} we only have to establish the formulas of $f$ and $v$. In our calculations below we need the following straight forward properties
	\begin{itemize}
		\item[(a)] $Q_\Delta$ symmetric,
		\item[(b)] $\int_0^\infty \phi_1(x,\vartheta)\phi_2(x,\vartheta)\mu_\vartheta(x)\diff x=0$,
		\item[(c)] $\int_0^\infty \phi_j(x,\vartheta)\mu_\vartheta(x)\diff x=0$,
		\item[(d)] $\int_0^\infty x^{2\eta} \mu_{\vartheta}(x)\diff x=\frac{\Gamma(\eta+\vartheta+1)}{\alpha^\eta\Gamma(\vartheta+1)}$ for $\eta \in \mathbb{N}$.
	\end{itemize}
	\textbf{Step 1:} According to \cite{kessler1999}, the formula for $f$ is given by 
	\begin{align*}
		f(\beta_1,\beta_2,\vartheta):=\sum_{i=1}^2 \int_0^\infty \int_0^\infty \frac{\partial }{\partial \vartheta} \beta_i(\vartheta)\left(\phi_i(x,\vartheta)-e^{-2\alpha\Delta} \phi_i(y,\vartheta)\right) Q_\Delta^\vartheta(\diff x,\diff y).
	\end{align*}
	We can easily calculate the two summands
	\begin{align*}
		&\int_0^\infty \int_0^\infty \frac{\partial }{\partial \vartheta} \beta_1(\vartheta)\left(\phi_1(x,\vartheta)-e^{-2\alpha\Delta} \phi_1(y,\vartheta)\right) Q_\Delta^\vartheta(\diff x,\diff y)\\
		&\stackrel{\textrm{(a)}}{=}(1-e^{-2\alpha\Delta})\int_{0}^\infty \int_0^\infty \frac{\partial }{\partial \vartheta} \beta_1(\vartheta)\phi_1(x,\vartheta)  Q_\Delta^\vartheta(\diff x,\diff y)\\
		&\stackrel{\textrm{(c)}}{=}(1-e^{-2\alpha\Delta}) \beta_1(\vartheta) \int_{0}^\infty \frac{\partial }{\partial \vartheta} \phi_1(x,\vartheta)  \mu_\vartheta(x)\diff x\\
		&=(1-e^{-2\alpha\Delta}) \beta_1(\vartheta) \int_{0}^\infty \frac{\alpha x^2}{(\vartheta+1)^2} \mu_\vartheta(x)\diff x\\
		&\stackrel{\textrm{(d)}}{=}\beta_1(\vartheta)\frac{ 1-e^{-2\alpha\Delta}}{\vartheta+1}
	\end{align*}
	and similarly
	\begin{align*}
		&\int_0^\infty \int_0^\infty \frac{\partial }{\partial \vartheta} \beta_2(\vartheta)\left(\phi_2(x,\vartheta)-e^{-4\alpha\Delta} \phi_2(y,\vartheta)\right) Q_\Delta^\vartheta(\diff x,\diff y)\\
		&=\beta_2(\vartheta)\frac{1-e^{-4\alpha\Delta}}{(\vartheta+1)(\vartheta+2)}.
	\end{align*}
	\textbf{Step 2:}  According to \cite{kessler1999}, we receive 
	\begin{align*}
		v(\vartheta)=\sum_{i,j=1}^{2} \beta_i(\vartheta)\beta_j(\vartheta)\alpha_{ij}(\vartheta)
	\end{align*}
	with 
	\begin{align*}
		\alpha_{ij}:=\int_0^\infty \int_0^\infty \left(\phi_i(y,\vartheta)-e^{-\lambda_i\Delta}\phi_i(x,\vartheta)\right)\hspace{-0.5mm}\cdot\hspace{-0.5mm}\left(\phi_j(y,\vartheta)-e^{-\lambda_j\Delta}\phi_j(x,\vartheta)\right)Q_\Delta(\diff x,\diff y).
	\end{align*}
	If we take a look at the proof of \ref{theo:asymptotic_estimator_mean_value}, we recognize the already calculated value  
	\begin{align*}
		\int_0^\infty \int_0^\infty \left(1-\frac{\alpha y^2}{\vartheta+1}-e^{-2\alpha\Delta}\left(1-\frac{\alpha x^2}{\vartheta+1}\right)\right)^2 Q_\Delta(\diff x,\diff y)=\frac{1-e^{-4\alpha\Delta}}{\vartheta+1}.
	\end{align*}
	For the remaining terms, it holds
	\begin{align*}
		&\int_0^\infty \int_0^\infty \left(\phi_1(y,\vartheta)-e^{-2\alpha\Delta}\phi_1(x,\vartheta)\right)\hspace{-0.5mm}\cdot\hspace{-0.5mm}\left(\phi_2(y,\vartheta)-e^{-4\alpha\Delta}\phi_2(x,\vartheta)\right)Q_\Delta(\diff x,\diff y)\\
		&\stackrel{\textrm{(a)},\textrm{(b)}}{=}-(e^{-2\alpha\Delta}+e^{-4\alpha \Delta})\int_0^\infty \int_0^\infty \phi_1(y,\vartheta)\phi_2(x,\vartheta ) Q_\Delta (\diff x, \diff y)\\
		&= -(e^{-2\alpha\Delta}+e^{-4\alpha \Delta})\int_0^\infty \int_0^\infty \left(1- \frac{\alpha y^2}{\vartheta+1}\right)p(x,y,\Delta)\diff y \phi_2(x,\vartheta ) \mu_\vartheta(x) \diff x\\
		&\stackrel{\textrm{(c)}}{=}(e^{-2\alpha\Delta}+e^{-4\alpha \Delta})\int_0^\infty \frac{\alpha}{\vartheta+1}\mean_{\mu_{\vartheta}}(X_\Delta^2 \cond{} X_0=x) \phi_2(x,\vartheta ) \mu_\vartheta(x) \diff x\\
		&=(e^{-2\alpha\Delta}+e^{-4\alpha \Delta})\int_0^\infty \left( \frac{\alpha}{\vartheta+1}x^2e^{-2\alpha \Delta}+1-e^{-2\alpha\Delta} \right) \phi_2(x,\vartheta ) \mu_\vartheta(x) \diff x\\
		&\stackrel{\textrm{(c)}}{=}\frac{\alpha(e^{-4\alpha\Delta}+e^{-6\alpha \Delta})}{\vartheta+1}\int_0^\infty x^2 \phi_2(x,\vartheta ) \mu_\vartheta(x) \diff x\\
		&=\frac{\alpha(e^{-4\alpha\Delta}+e^{-6\alpha \Delta})}{\vartheta+1}\int_0^\infty \left( x^2-\frac{2\alpha x^4}{\vartheta+1} +\frac{\alpha^2x^6}{(\vartheta+1)(\vartheta+2)} \right) \mu_\vartheta(x) \diff x\\
		&\stackrel{\textrm{(d)}}{=}\frac{(e^{-4\alpha\Delta}+e^{-6\alpha \Delta})}{\vartheta+1}\left(\vartheta+1-2(\vartheta+2)+\vartheta+3\right)\\
		&=0
	\end{align*}
	and by similar calculations, we get
	\begin{align*}
		&\int_0^\infty \int_0^\infty \left(\phi_2(y,\vartheta)-e^{-4\alpha\Delta}\phi_2(x,\vartheta)\right)^2Q_\Delta(\diff x,\diff y)\\
		&\stackrel{\textrm{(a)}}{=}(1+e^{-8\alpha \Delta})\int_0^\infty \phi_2^2(x,\vartheta)\mu_{\vartheta}(x)\diff x -2e^{-4\alpha\Delta} \int_0^\infty \int_0^\infty \phi_2(x,\vartheta)\phi_2(y,\vartheta) Q_\Delta(\diff x, \diff y)\\
		&=\frac{2-2^{-8\alpha\Delta}}{(\vartheta+1)(\vartheta+2)}.
	\end{align*}
	\end{proof}
	Our aim is now to find $\beta_i$s, which lead to the smallest asymptotic variance as $\alpha \Delta\to \infty$. Therefore, we define for fixed $\vartheta \in \Theta$ the approximated functions
\begin{align*}
	&\widetilde{v}(\beta_1,\beta_2):=\frac{\beta_1^2(\vartheta)}{\vartheta+1}+\frac{2\beta_2^2(\vartheta)}{(\vartheta+1)(\vartheta+2)},\\
	&\widetilde{f}(\beta_1,\beta_2):=\frac{\beta_1(\vartheta)}{\vartheta+1}+\frac{\beta_2(\vartheta)}{(\vartheta+1)(\vartheta+2)},
\end{align*}
for which 
\begin{align*}
	\lim\limits_{\alpha\Delta\to \infty}\left| \frac{v(\vartheta)}{f^2(\vartheta)} - \frac{\widetilde{v}(\vartheta)}{\widetilde{f}^2(\vartheta)} \right|=0 
\end{align*}
holds. 
This property justifies the search for the global minimum of 
\begin{align*}
	(\beta_1,\beta_2)\mapsto \frac{\widetilde{v}(\beta_1,\beta_2)}{\widetilde{f}^2(\beta_1,\beta_2)}.
\end{align*}
To establish the minimum we first simplify the function
\begin{align*}
	\frac{\widetilde{v}(\beta_1,\beta_2)}{\widetilde{f}^2(\beta_1,\beta_2)}&= \frac{\frac{\beta_1^2}{\vartheta+1}+\frac{2\beta_2^2}{(\vartheta+1)(\vartheta+2)} }{\left( \frac{\beta_1}{\vartheta+1}+\frac{\beta_2}{(\vartheta+1)(\vartheta+2)}\right)^2}\\
	&=(\vartheta+1)(\vartheta+2)\frac{(\vartheta+2)\beta_1^2+2\beta_2^2 }{\left( (\vartheta+2)\beta_1+\beta_2\right)^2}
\end{align*}
and determine the first derivatives
\begin{align*}
	\frac{\diff }{\diff \beta_1}\frac{\widetilde{v}(\beta_1,\beta_2)}{\widetilde{f}^2(\beta_1,\beta_2)}&= 
	2(\vartheta+1)(\vartheta+2)^2\frac{\beta_1\beta_2-2\beta_2^2}{ \left( (\vartheta+2)\beta_1+\beta_2\right)^3},
\end{align*}
\begin{align*}
	\frac{\diff }{\diff \beta_2}\frac{\widetilde{v}(\beta_1,\beta_2)}{\widetilde{f}^2(\beta_1,\beta_2)}&= 2(\vartheta+1)(\vartheta+2)^2\frac{2\beta_1\beta_2-\beta_1^2 }{\left( (\vartheta+2)\beta_1+\beta_2\right)^3}.
\end{align*}
Taking into account the properties of the $\beta_i$s in Theorem \ref{theo:asymptotic_variance}, we get as possible minima $\beta_1=2\beta_2\not=0$ with value
\begin{align*}
	\frac{\widetilde{v}(2\beta_2,\beta_2)}{\widetilde{f}^2(2\beta_2,\beta_2)}&=\frac{2(\vartheta+1)(\vartheta+2) }{2\vartheta+5}.
\end{align*}	
In order to check, if we indeed have minima, we consider $\beta_1\not = 2\beta_2$
and see
\begin{align*}
	&\frac{\widetilde{v}(\beta_1,\beta_2)}{\widetilde{f}^2(\beta_1,\beta_2)}-\frac{2(\vartheta+1)(\vartheta+2) }{2\vartheta+5}=(\vartheta+1)(\vartheta+2)^2\frac{(\beta_1-2\beta_2)^2}{(2\vartheta+5)\left( (\vartheta+2)\beta_1+\beta_2\right)^2}>0.
\end{align*}
Hence, these critical points are global minima.
Finally, we may specify the improvement of the asymptotic variance 
\begin{align*}
	&\vartheta+1-\frac{2(\vartheta+1)(\vartheta+2) }{2\vartheta+5} =(\vartheta+1)\frac{2\vartheta+5-2(\vartheta+2)}{2\vartheta+5} =\frac{\vartheta+1}{2\vartheta+5}>0
\end{align*}
if we consider the asymptotic behaviour $\alpha\Delta\to \infty$. Hence, we see that relative improvement compared to $\vartheta+1$ is $\frac{1}{2\vartheta+5}$ and decreases as $\vartheta$ increases. However, for the boundary case $\vartheta=-1/2$ we get an improvement of $25\%$. For the case $\vartheta=0$, which for a Dunkl process separates between finite and infinite jump activity, we still get an improvement of $20\%$.\\

{\bf{Acknowledgements.}}\\
The financial support of the DFG-GRK 2131 is gratefully acknowledged.

\end{document}